\documentclass[a4paper,11pt]{amsart}
\usepackage[margin=3.5cm]{geometry}
\usepackage{setspace}
\usepackage{amssymb,enumerate,color,fancyvrb,prooftree}
\usepackage{prooftree}
\usepackage{multicol}
\usepackage[pagebackref=true]{hyperref}

\def\oo{\"{o}}
\theoremstyle{plain}
\newtheorem{theorem}{Theorem}[section]

\newtheorem{lemma}[theorem]{Lemma}
\newtheorem{corollary}[theorem]{Corollary}

\theoremstyle{definition}
\newtheorem{definition}[theorem]{Definition}

\newtheorem{remark}[theorem]{Remark}

\newcommand{\N}{\mathbb{N}}

\def\on{\mathrm{On}}
\def\seq{\mathrm{\N^{<\N}}}

\def\klo{\mathcal O}

\def\vp{\varphi}

\DeclareMathOperator{\thm}{\mathrm{Thm}}
\newcommand{\md}[1]{ \mathrm{mod}\ #1}
\DeclareMathOperator{\weak}{\mathrm{weak}}

\newcommand{\kl}[1]{\{#1\}} 

\def\xvdash#1#2{
	\setbox1=\hbox{\kern1.5pt$\scriptstyle#2$}
	\ifx\zeichen\empty\setbox0=\hbox to .75em{}\else\setbox0=\hbox
	{\kern1.5pt$\scriptstyle#1$}\fi
	\dimen1=\dp0 \ifdim \dimen1=0pt
	\advance \dimen1 by 1.5ex \else \advance \dimen1 by 1.2ex
	\fi\dimen3=2ex\dimen4=.5ex\ifdim \wd0<\wd1 \dimen2=\wd1 \else \dimen2=\wd0
	\fi\hbox{\hskip.5em$\vrule height\dimen3
		depth\dimen4\raise\dimen1\copy0\hskip-1\wd0
		\lower\ht1\copy1\hskip-1\wd1\vrule width\dimen2 height.7ex depth-.6ex
		\hskip3pt minus1.5pt$}}
\def\rul#1#2#3{\prooftree     #1 \justifies #2 \using{#3} \endprooftree}
\def\ptree#1{\prooftree  #1 \endprooftree}
\def\j{\justifies}
\def\ax{\text{Ax}}

\def\cut{cut}
\def\orule{\omega\text{-}rule}
\def\rep{rep}
\def\orep{\omega\text{-}rep}

\def\Crec{\mathcal{C}_{rec}}
\def\Cpr{\mathcal{C}_{prec}}
\def\Irec{\mathcal{I}_{rec}}
\def\Ipr{\mathcal{I}_{prec}}
\newcommand{\ha}{{\sf HA}}                         
\newcommand{\pa}{\ensuremath{\mathsf{PA}}}         

\def\lim{\mathrm{Lim}}

\newcommand{\conc}{{{}^\smallfrown}}            

\DeclareMathOperator{\Rule}{\mathrm{Rule}}
\DeclareMathOperator{\End}{\mathrm{End}}

\title{Completeness of  the primitive recursive $\omega$-rule}
\thanks{The author was supported by John Templeton Foundation (\lq\lq A new dawn of Intuitionism: Mathematical and Philosophical advances\rq\rq, grant ID 60842).}
\author[Frittaion]{Emanuele Frittaion}
\address{School of Mathematics, University of Leeds, UK}
\keywords{Shoenfield's completeness theorem, $\omega$-rule, primitive recursion theorem, Kleene's $\klo$, $\Pi^1_1$ completeness}

\begin{document}

\subjclass[2010]{03D20, 03D70, 03F03} 

\begin{abstract} 
Shoenfield's completeness theorem (1959) states that every true first order arithmetical sentence has a recursive $\omega$-proof encodable by using recursive applications of the $\omega$-rule. For a suitable encoding of Gentzen style $\omega$-proofs, we show that Shoenfield's completeness theorem applies to cut free $\omega$-proofs encodable by using primitive recursive applications of the $\omega$-rule. We also show that the set of codes of $\omega$-proofs, whether it is based  on recursive or primitive recursive applications of the $\omega$-rule, is $\Pi^1_1$ complete. The same $\Pi^1_1$ completeness results apply to codes of cut free $\omega$-proofs.
\end{abstract}

\maketitle
\tableofcontents

\section{Introduction}
The $\omega$-rule 
\[   \rul{A( 0)\ \ A( 1)\ \  \cdots\ \  A(n)\ \   \cdots }{\forall xA(x)}{\orule} \]
has been subject to  proof theoretic investigations since the 1930s. For a detailed account on the $\omega$-rule, including  historical references,  see Sundholm \cite{S83}. The most important application related to the $\omega$-rule is probably due to Sch\"{u}tte \cite{S50} with his ordinal analysis of Peano arithmetic (subsequently extended to ramified analysis) via cut elimination for $\omega$-arithmetic. Another milestone is Shoenfield's completeness theorem \cite{Sho59}, which asserts that the recursive $\omega$-rule is  complete for  true  arithmetical statements. (That this holds good for the  unrestricted $\omega$-rule is almost trivial, by induction on the build up of a true sentence.)
Shoenfield's result is  essential for another completeness result, this time due to Feferman \cite{F62}, asserting the completeness for true arithmetical statements of certain transfinite recursive sequences (progressions) of arithmetic theories based on iterated reflection principles. 


An $\omega$-proof is naturally represented as either a countable transfinite sequence of formulas (Hilbert style) or a countable well founded  infinite tree of sequents (Gentzen style). In both cases, we are dealing with a second order object (countable set).  Shoenfield's completeness theorem refers to an inductive definition of  G\"{o}del numbers of Hilbert style $\omega$-proofs. The coding is based on recursive applications of the $\omega$-rule. Roughly,  the index of a recursive enumeration  $f$ of G\"{o}del numbers  $f(n)$ of  $\omega$-proofs of $A(n)$  is the G\"{o}del number of an $\omega$-proof of  $\forall xA(x)$. The coding  is \emph{local} as opposed to \emph{global}. The latter involves indices of recursive $\omega$-proofs. We shall refer to local codes based on recursive applications of the $\omega$-rule as to recursive local codes. Then Shoenfield's completeness theorem  says that an arithmetical sentence $A$ is true if and only if there is a recursive local code of an $\omega$-proof of $A$.  As expected, a recursive local code is the code of a recursive $\omega$-proof. Indeed, an $\omega$-proof is recursive  if and only if it has a recursive local code.

It is natural to ask whether, under a suitable local coding of Gentzen style $\omega$-proofs, the same completeness result applies to codes of cut free $\omega$-proofs. For mere convenience, we will consider Tait one-sided sequent calculus. Sequents are finite sets of sentences $\Gamma$ to be interpreted disjunctively.  Let $\Crec$ be the set of recursive local  codes of $\omega$-proofs based on Tait one-sided sequent calculus (Definition \ref{rec codes}).  By implementing the \emph{proof-search} completeness proof of $\omega$-arithmetic, one can see that this is the case. Therefore an arithmetical sentence $A$ is true if and only if there is a local code in $\Crec$ of a cut free $\omega$-proof of $A$  (cf.\ Theorem \ref{shoenfield}). 

It is also natural to ask what the complexity of $\Crec$ is. We show, not surprisingly, that this set is as complicated as possible, that is, $\Pi^1_1$ complete (Theorem \ref{thm rec complete}).  The author was not able to find this result in the literature.  It turns out that a very small cut free fragment of $\Crec$ is already $\Pi^1_1$ complete. Indeed, we just need the $\omega$-rule.  

Note that by implementing the \emph{proof-search} completeness  proof of higher order $\omega$-arithmetic (for languages with at least second order objects) with respect to $\omega$-models, one obtains in a rather neat way that the set of indices (global codes) for  primitive recursive higher order $\omega$-proofs is $\Pi^1_1$ complete (cf.\ Girard \cite{G87}).  More precisely, given a $\Pi^1_1$ formula $A(x)$, there exists a primitive recursive function $f$ such that $A(\bar n)$ holds iff $f(n)$ is the  index  of a primitive recursive  cut free $\omega$-proof of $A(\bar n)$. 
(One has to add logical axioms of the form $\Gamma,t\in X,s\notin X$, where $s$ and $t$ are closed terms with the same numerical value.)

What about the primitive recursive $\omega$-rule? Let $\Cpr$ be the set of local codes based on primitive recursive applications of the $\omega$-rule. We shall refer to local codes in $\Cpr$ as to primitive recursive local codes. By applying the conversion procedures between recursive and primitive recursive codes from Sundholm's \cite{S83}, one obtains, by a detour to global codes,  that an arithmetical sentence $A$ is true if and only if there is a primitive recursive local code of a cut free $\omega$-proof of $A$ containing vacuous applications of the $\omega$-rule
\[ \rul{\\ \ \Gamma\ \  \cdots \ \  \Gamma \ \  \cdots}{\Gamma}{}. \]
However, the question remains if one can obtain a primitive recursive local code of an $\omega$-proof without vacuous rules. The first approach consists in eliminating the use of vacuous rules in favour of cuts in the conversion between local and global codes (Theorem \ref{pr index to pr code} is a case in point). This leads to the completeness for true arithmetic of $\Cpr$, but it does not answer the question for the cut free fragment of $\Cpr$. We will answer in the affirmative by providing  a many-one reduction of $\Crec$ to $\Cpr$ that preserves all logical rules and in particular preserves cut free $\omega$-proofs (Theorem \ref{thm pr complete}). As a byproduct, we obtain that also (the cut free fragment of) $\Cpr$ is $\Pi^1_1$ complete.

It is important to stress that the correspondence between  $\omega$-proofs and local codes breaks down for subrecursive applications of the $\omega$-rule.  A primitive recursive local code, namely  a local code in $\Cpr$, need not correspond to a primitive recursive $\omega$-proof, and vice versa. See Section \ref{index vs code} for a detailed comparison between local and global codes.

\subsection{First order classical $\omega$-arithmetic}
We consider a Tait one-sided sequent calculus. Formulas are in negation normal form (formulas are obtained from literals, consisting of atomic and negated atomic formulas in the language $0,1,+,\times,=$, by using $\land,\lor,\forall,\exists$). Sequents $\Gamma,\Delta,\ldots$ are finite sets of sentences. The interpretation of $\Gamma$ is $\bigvee\Gamma$. \smallskip
\begin{center} \textcolor{blue}{Logical axioms} \end{center} \smallskip 
\[ \Gamma, A \ \ \ \ \ A \text{ closed true literal}   \]
\smallskip
\begin{center} \textcolor{blue}{Logical rules}  \end{center} \smallskip
\[
\begin{array}{cc}
\begin{prooftree}
\Gamma, A \ \ \ \Gamma, B 
\justifies
\Gamma, A\land B \using{\land}
\end{prooftree}
&    
\begin{prooftree}
\Gamma, A, B
\justifies
\Gamma,A\lor B
\using{\lor}
\end{prooftree}
\\[8mm]
\begin{prooftree}
\{\Gamma, A(\bar n)\colon n\in\omega\}\ 
\justifies
\Gamma, \forall xA(x)
\using{\orule}
\end{prooftree} \qquad
& \qquad  
\begin{prooftree}
\Gamma, A(\bar n)  
\justifies
\Gamma, \exists xA(x)  
\using{\exists}
\end{prooftree}
\\[8mm]
\end{array} \]

\[   \begin{prooftree}
\Gamma, C
\quad
\Gamma, \neg C 
\justifies 
\Gamma  \using{\cut}
\end{prooftree}\]
\smallskip

An $\omega$-proof   is a well founded tree consisting of sequents, which is \emph{locally correct} with respect to axioms and inference rules.  That is, every top-node sequent is an axiom and  every branching in the tree corresponds to a correct inference. 

More formally, we define an $\omega$-proof as follows. Let $\seq\subseteq\N$ be a (primitive recursive) set of G\oo del numbers of finite sequences of natural numbers.
We denote sequences by $\sigma,\tau,\ldots$  Let $\sigma\conc\tau$ be the concatenation of $\sigma$ with $\tau$. Write $\sigma\conc i$ for $\sigma\conc \langle i\rangle$. A tree $T$ is a subset of $\seq$ closed under initial segments. A tree $T$ is well founded if there is no function $f\colon \N\to \N$ such that $\bar fn\in T$ for every $n$, where $\bar fn=\langle f(0),\ldots,f(n-1)\rangle$. For convenience,  assume $0\notin\seq$. 

\begin{definition}\label{omega proof}
	A total function $\pi\colon \N\to\N$ is an $\omega$-proof if
	\begin{itemize}
		\item $T_\pi=\{ \sigma\in\seq\colon \pi(\sigma)\in\seq\}$ is a well founded tree, 
		\item $\pi(\sigma)\notin\seq$ implies $\pi(\sigma)=0$,
		\item $\sigma\conc j \in T_\pi$ implies $\sigma\conc i\in T_\pi$ for every $i<j$,
		\item if $\sigma\in T_\pi$ is a top-node, then $\pi(\sigma)=\langle \ax,\Gamma\rangle$, where $\Gamma$ is an axiom, $\Rule(\pi(\sigma))=\ax$, and $\End(\pi(\sigma))=\Gamma$,
		\item if $\sigma\in T_\pi$ is not a top-node, then $\pi(\sigma)=\langle R,\Gamma, A\rangle$, where $R\in\{\land,\lor,\omega,\exists,\cut\}$, $\Rule(\pi(\sigma))=R$,   \[ \End(\pi(\sigma))=\begin{cases} \Gamma\cup A & \text{ if }R\neq\cut \\ \Gamma & \text{ otherwise} \end{cases},\]  and
		\[ \rul{\{\End(\pi(\sigma\conc  i)) \colon \sigma\conc i\in T_\pi\}}{\End(\pi(\sigma))}{R} \]
		is a correct inference.
	\end{itemize}	
\end{definition}
For technical reasons,  it is convenient to encode the side formulas (context) $\Gamma$ and the main formula (or the cut formula) $A$  of the last rule $R$.   Therefore the end sequent is $\Gamma,A$ if $A$ is the main formula, and $\Gamma$ otherwise. The idea is that we want to read off the rule and the premises of an inference from the conlusion. For instance, in the case of the $\omega$-rule, if we simply encode the conclusion by a sequence of the form $\langle \omega, \Gamma,\forall xA(x)\rangle$, where $\omega$ is a conventional name for the $\omega$-rule, the second component is the end sequent, and the third component is the main formula, we cannot decide what the premises $(\Gamma_n)_{n\in\omega}$  are.  In fact, either $\Gamma_n=\Gamma,A(\bar n)$ for all $n$ or $\Gamma_n=\Gamma,\forall x A(x),A(\bar n)$ for all $n$.

\subsection{Primitive recursion theorem}
Let $\{\kl e\colon e\in\omega\}$ be a standard numbering of the partial recursive functions. We will use the recursion theorem  in the following form.\footnote{Kleene in \cite[p.\ 75]{K58} states the primitive recursion theorem with respect to a standard numbering $\{[e]\colon e\in\omega\}$ of all primitive recursive functions. That is, for every primitive recursive function $g(x,x_1,\ldots,x_n)$ there is an index $e$ such that for all $x_1,\ldots,x_n$
	\[    [e](x_1,\ldots,x_n)=g(e,x_1,\ldots,x_n). \]}
\begin{theorem}
	For every primitive recursive function $g(x,x_1,\ldots,x_n)$ there is a primitive recursive function $f$ of index $e$ such that for all $x_1,\ldots,x_n$
	\[    \kl e(x_1,\ldots,x_n)=g(e,x_1,\ldots,x_n). \]
\end{theorem}
\begin{proof}
By the recursion theorem. 
\end{proof}

\subsection{s-m-n theorem} We will repeatedly use the s-m-n theorem in the following form.
\begin{theorem}
Let $\vp(x,x_1,\ldots,x_n)$ be a partial recursive function. Then there is a primitive recursive function $f$ such that 
\[   f(x_1,\ldots,x_n)=\Lambda x.\vp(x,x_1,\ldots,x_n). \]
That is, for all $x_1,\ldots,x_n$
\[  \kl{f(x_1,\ldots,x_n)}(x)\simeq \vp(x,x_1,\ldots,x_n). \]
\end{theorem} 




\subsection{Kleene's system $\klo$}
The set $\klo$ of notations for constructive ordinals is given by the set $\{ a\colon a<_\klo 2^a\}$ where the relation $<_\klo$ is inductively defined by:
\begin{itemize}
	\item $1<_\klo 2$;
	\item if $a<_\klo b$ then $b<_\klo 2^b$;
	\item if $\kl e$ is total and $\kl e(n)<_\klo \kl e(n+1)$ for every $n$, then  $\kl e(n)<_\klo 3\cdot 5^e$ for every $n$;
	\item if $a<_\klo b$ and $b<_\klo c$, then $a<_\klo c$.
\end{itemize}
The relation $<_\klo$ is a well founded partial ordering and for every $b\in\klo$ the initial interval $\{a\colon a<_\klo b\}$ is linearly ordered by $<_\klo$.  The set $\klo$ is  $\Pi^1_1$ complete (Kleene).

We will need another characterization of $\klo$. Inductively define a recursively enumerable relation $<'_\klo$ by:
\begin{itemize}
	\item $1<'_\klo 2$;
	\item  if $a<'_\klo b$ then $b<'_\klo 2^b$; 
	\item if $\kl e(n)$ is defined then $\kl e(n)<'_\klo 3\cdot 5^e$;
	\item if $a<'_\klo b$ and $b<'_\klo c$, then $a<'_\klo c$;
\end{itemize}
Finally, inductively define $\klo$  by:
\begin{itemize}
	\item $1\in\klo$;
	\item  if $a\in\klo$ then $2^a\in\klo$;
	\item if $\kl e$ is total and $\kl e(n)<'_\klo \kl e(n+1)$ for every $n$, then $3\cdot 5^e\in\klo$.
\end{itemize}
Note that the recursively enumerable relation $<'_\klo$ coincides with $<_\klo$ on $\klo$.

Define the length (order type) $|a|\in\on$ of $a\in\klo$ by 
\[    |1|=0, \]
\[    |2^a|=|a|+1, \]
\[    |3\cdot 5^e|=\sup\{ |\kl e(n)|\colon n\in\omega\}. \]

\section{Shoenfield recursive $\omega$-rule}
We describe a notation system for $\omega$-proofs based on Tait sequent calculus (see Definition \ref{omega proof}).\footnote{Cf.\ Feferman \cite[Definition 5.7, p.\ 302]{F62}, where codes are based on a Hilbert style presentation of first order $\omega$-arithmetic.}  The set $\Crec$ of recursive local codes of $\omega$-proofs is a typical example of inductive definition (cf.\ $\klo$). It follows that $\Crec$ is $\Pi^1_1$.
Every code is of the form $\langle R,\Gamma,A,a_1,\ldots,a_n\rangle$, where $\Gamma$ is the context and $A$ is the main formula (or the cut formula)  of the last rule $R$. We are assuming a primitive recursive G\oo del numbering of formulas etc. To avoid too heavy notation, we omit corner quotes throughout.  The primitive recursive functions $\End$ (for end sequent) and $\Rule$ (for last rule) are defined by  $\End(a)=(a)_1$ if $\Rule(a)=(a)_0=\cut$ and $(a)_1\cup(a)_2$ otherwise. 

\begin{definition}\label{rec codes}
	$\Crec$ is the smallest subset of $\N$ such that:
	\begin{itemize}
		\item[(axiom)]  for every axiom $\Gamma$, $\langle\ax, \Gamma\rangle \in \Crec$;
		\item[($\land$)] if $a,b\in \Crec$, $\End(a)=\Gamma,A$, $\End(b)=\Gamma,B$, then 
		$\langle \land,\Gamma, A\land B, a,b\rangle\in \Crec$;
		\item[($\lor$)] if $a\in \Crec$ with $\End(a)=\Gamma,A,B$ then $\langle\lor, \Gamma,A\lor B, a \rangle\in \Crec$;
		\item[($\omega$-rule)] if $\kl e$ is total,  $\kl e(n)\in \Crec$ and  $\End(\kl e(n))=\Gamma,A(\bar n)$ for every $n$, then \newline
		$\langle \omega, \Gamma,\forall xA(x),e\rangle\in \Crec$;
		\item[($\exists$)] if $a\in \Crec$ and $\End(a)=\Gamma, A(\bar n)$ then $\langle \exists, \Gamma,\exists xA(x),a\rangle \in \Crec$;
		\item[(cut)] if $a,b\in \Crec$, $\End(a)=\Gamma,C$ and $\End(b)=\Gamma,\neg C$, then $\langle \cut, \Gamma,C, a,b\rangle\in \Crec$.
	\end{itemize}
\end{definition}

Let $a \in \Crec$. Define the length $|a|\in\on$ of $a$ by
\begin{align*}
|\langle \ax,\Gamma\rangle|&=0,\\
|\langle R,\Gamma,A,a_1,a_2,\ldots,a_n\rangle|&=\sup\{|a_1|,|a_2|,\ldots,|a_n|\}+1, \\
|\langle\omega, \Gamma,\forall xA(x),e\rangle|&= \sup\{|\kl e(n)|+1\mid n\in\omega\}.
\end{align*}
For $a\in\Crec$, define the $\omega$-proof $\pi(a)$ by induction on the code $a$. For instance, if $a=\langle \omega,\Gamma,\forall xA(x),e\rangle$, then 
\[         \pi(a)(\langle\rangle)=\langle \omega,\Gamma,\forall xA(x)\rangle, \]
\[         \pi(a)(\langle n\rangle\conc\sigma)= \pi(\kl e(n))(\sigma). \] 
Note that $\pi(a)$ is an $\omega$-proof of $\End(a)$ in the sense of Definition \ref{omega proof}. Hence,  $\End(a)=\End(\pi(a)(\langle\rangle))$.  Also,   $|a|=|\pi(a)|$, where the length of the $\omega$-proof $\pi(a)$ is the length (aka height) of the well founded tree $T_{\pi(a)}$. \smallskip

It turns out that an $\omega$-proof $\pi$ is recursive if and only if there is recursive local code $a$ such that $\pi(a)=\pi$.  For more information on local and global codes,  see Section \ref{index vs code}. \smallskip

Let $\Crec^-$ be the cut free fragment of $\Crec$. Set $\thm(\Crec)=\{A\colon \End(a)=A \text{ for some }a\in \Crec\}$. Similar definitions apply to all proof notation systems considered in this paper.

\begin{theorem}[Shoenfield]\label{shoenfield}
For every true arithmetical sentence $A$ there is recursive local code $a\in\Crec^-$ such that $\pi(a)$ is a primitive recursive cut free $\omega$-proof of $A$. Therefore, $\thm(\Crec)=\thm(\Crec^-)$  coincides with the set of true arithmetical sentences.
\end{theorem}

\begin{proof}[Proof sketch]
We implement the proof-search construction based on  Sch\"{u}tte's deduction chains \cite{S77}. The very same idea applies to different Gentzen style proof systems. In our context, given a sequent $\Gamma$, one can build primitive recursively in the parameter $\Gamma$ a function $\pi\colon\N\to\N$ such that $\Gamma$ is true if and only if $\pi$ is an $\omega$-proof of $\Gamma$ in the sense of Definition \ref{omega proof}. We will refer to $\pi$ as to the  Sch\"{u}tte tree of $\Gamma$ and write $\pi(\Gamma)$. Indeed, there is a primitive recursive function $\pi(x,y)$ such that  $\pi(\Gamma)=\lambda x.\pi(\Gamma,x)$ for every sequent $\Gamma$.  Moreover,  $\Gamma$ is true if and only if  $\pi(\Gamma)$ is a cut free $\omega$-proof of $\Gamma$  if and only if $T_{\pi(\Gamma)}$ is well founded. In fact, $\pi(\Gamma)$ always defines a locally correct, although not necessarily well founded, $\omega$-proof. 

We want to define a partial recursive function $\psi$ such that $\psi(\Gamma)$ is defined  with  $\psi(\Gamma)\in \Crec$ and $\End(\psi(\Gamma))=\Gamma$,  whenever  $\Gamma$ is a true sequent. By construction, $\psi(\Gamma)$ will code a primitive recursive cut free $\omega$-proof of $\Gamma$, namely the Sch\"{u}tte tree of $\Gamma$.

For this construction one considers ordered sequents. An ordered sequent has the form  $\Gamma$ or $\Gamma,A,\Delta$ where $\Gamma$ consists of literals and $A$ is not a literal.\footnote{Strictly speaking, the Sch\"{u}tte tree consists of ordered sequents. Clearly, one can primitive recursively in the function $\psi(\Gamma)$ forget the order and hence obtain a primitive recursive tree of (unordered) sequents.}

By the recursion theorem, there is a partial recursive function  $\psi$ with index $e$  satisfying the following recursive equations.
\begin{align*}
\kl e(\Gamma)&\simeq \langle \ax,\Gamma\rangle   \ \ \ \Gamma \text{ axiom}\\
\kl e(\Gamma,A\land B,\Delta)&\simeq\langle\land, \Gamma\cup \Delta,A\land B,\kl e(\Gamma,A,\Delta), \kl e(\Gamma,B,\Delta) \rangle\\
\kl e(\Gamma,A\lor B,\Delta)&\simeq \langle\lor, \Gamma\cup\Delta, A\lor B, \kl e(\Gamma,A,B,\Delta) \rangle\\ 
\kl e(\Gamma,\forall x A(x),\Delta)&\simeq\langle\omega,\Gamma\cup\Delta,\forall xA(x), \Lambda n. \kl e(\Gamma,A(\bar n),\Delta)\rangle \\
\kl e(\Gamma,\exists xA(x),\Delta)&\simeq \langle \exists, \Gamma\cup\Delta,\exists xA(x),\kl e(\Gamma,A(\bar n),\Delta,\exists xA(x))\rangle 
\end{align*}
In the  clause for $\exists$, $n$ is least such that $A(\bar n)$ does not occur in the sequent $\Gamma$. In the clause for the $\omega$-rule, $\Lambda n.\kl e(\Gamma,A(\bar n),\Delta)$ is an index of the partial recursive function $\lambda n.\kl e(\Gamma,A(\bar n),\Delta)$. Such index can be obtained in a primitive recursive way from $e$ and $\Gamma,\forall xA(x),\Delta$.

Suppose $\Gamma$ is true. The Sch\"{u}tte tree of $\Gamma$ yields a primitive recursive cut free $\omega$-proof $\pi(\Gamma)$ of $\Gamma$. In particular, it is well founded. By induction on the length of the tree,  show that $\psi(\Gamma)$ is defined and $\psi(\Gamma)\in\Crec^-$ is a code of $\pi(\Gamma)$.
\end{proof}

Shoenfield  gives the bound $\omega^\omega$ on the length of a recursive $\omega$-proof of a true arithmetical statement.\footnote{Cf.\ Feferman \cite[Theorem 5.13]{F62}. Every true sentence is provable in some $\pa_d$ with $d\in\klo$ and $|d|<\omega^{\omega^\omega}$.  Fenstad \cite{Fen68} improves Feferman's bound to $\omega^\omega$.}  It is not difficult to see that the length of the Sch\"{u}tte tree of a sequent containing a true  sentence of rank (degree) $n$ is $<\omega\cdot (n+1)$ (cf.\ Franz\'{e}n \cite{Fra04}). Therefore one obtains the bound $\omega^2$ for recursive Tait style $\omega$-proofs of true sentences. Note that the precise bound may depend  on the particular choice of the proof system.  Of course,  on account of Tarski's theorem, one cannot equip such recursive $\omega$-proofs with recursive ordinal bounds (that is, one cannot obtain a recursive $\omega$-proof where every sequent  appearing in it comes with an ordinal notation  for an ordinal $<\alpha$ bounding the length of the corresponding $\omega$-subproof, where $\alpha$ is a fixed recursive ordinal).

\subsection{$\Pi^1_1$ completeness of $\Crec$}

We show that $\klo$ is many-one reducible to $\Crec$ and thus $\Crec$ is $\Pi^1_1$ complete. The same reduction will witness that $\Crec^-$ is also $\Pi^1_1$ complete.

\begin{theorem}\label{thm rec complete}
There exists a primitive recursive function $f$ such that $a\in \klo$ if and only if  $f(a)\in\Crec$. Therefore $\Crec$ is $\Pi^1_1$ complete. Indeed, $f(a)\in\Crec^-$ for every $a\in\klo$ and so $f$ also witnesses a many-one reduction of $\klo$ to $\Crec^-$.
\end{theorem}
\begin{proof}
The idea is to manufacture artificially long $\omega$-proofs of the sentence $\forall x(x=x)$. Successor and limit stages are treated similarly. Essentially, $f(a)$ for $a\in\klo$ will be the code of an $\omega$-proof whose last inference is 
	\[ \rul{\{\forall xA(x),A(\bar n)\colon n\in\omega\}}{\forall xA(x)}{\orule}, \]
where $A(x)$ is $x=x$.
In both cases we make use of a primitive recursive function $\weak$ (for weakening) such that for every $a\in\Crec$ and for every sequent $\Delta$, $\weak(a,\Delta)\in\Crec$ with $\End(\weak(a,\Delta))=\End(a)\cup\Delta$. As usual, this function is easily defined by  the recursion theorem and seen to be primitive recursive by direct inspection. 
	\begin{align*}
	\weak(\langle \ax, \Gamma\rangle,\Delta)&= \langle \ax,\Gamma\cup\Delta\rangle   \\
	\weak(\langle R, \Gamma,B,a_1,\ldots,a_n\rangle,\Delta)&=\langle R, \Gamma\cup\Delta,B,\weak(a_1,\Delta),\ldots,\weak(a_n,\Delta)\rangle \\
	\weak(\langle \omega, \Gamma,\forall xA,e\rangle,\Delta)&= \langle \omega, \Gamma\cup\Delta,\forall xA,\Lambda n.\weak(\kl e(n),\Delta)\rangle
	\end{align*}
For the proof at hand, though, we need  more control on $\weak$. Namely, we want $\weak$ so that $\weak(a,\Delta)\in\Crec$ implies $a\in\Crec$. For this purpose, define $\weak$ such that 
\begin{align*}
\weak(\langle \ax, \Gamma\rangle,\Delta)&= \langle \ax,\Gamma\cup\Delta\rangle \qquad \text{ if $\Gamma$ is an axiom},\\
\weak(\langle \land,\Gamma,A\land B,a,b\rangle,\Delta)&= \langle \land,\Gamma\cup \Delta,A\land B,\weak(a,\Delta),\weak(b,\Delta)\rangle,     \\ 
& \qquad \text{ if } \End(a)=\Gamma,A \text{ and } \End(b)=\Gamma,B, \\
\weak(\langle\lor,\Gamma,A\lor B,a\rangle,\Delta) &= \langle \lor,\Gamma\cup\Delta,A\lor B,\weak(a,\Delta)\rangle,   \\
& \qquad \text{ if } \End(a)=\Gamma,A,B, \\
\weak(\langle \omega,\Gamma,\forall xA(x),e\rangle,\Delta) &= \langle\omega,\Gamma\cup\Delta,\forall xA(x),f(e)\rangle, \\
\weak(\langle \exists,\Gamma,\exists xA(x),a\rangle,\Delta) &= \langle \exists,\Gamma\cup\Delta,\weak(a,\Delta)\rangle, \\
& \qquad \text{ if } \End(a)=\Gamma,A(\bar n) \text{ for some } n. 
\end{align*}
Here, the function $f$ in the clause for the $\omega$-rule  is defined by
\[   \kl{f(e)}(n) \simeq \begin{cases}
 \weak(\kl e(n),\Delta) & \text{ if }  \End(\kl e(n))=\Gamma,A(\bar n);\\
 \text{undefined}   & \text{ otherwise.}
\end{cases} \]

The recursion theorem gives us a primitive recursive function $f$ such that	
	\begin{align*}
	f(1)&= d, \\
	f(2^a)&=\langle\omega, \forall xA(x), \forall xA(x),g(a)\rangle, \\
	f(3\cdot 5^e)&= \langle \omega, \forall xA(x),\forall xA(x),h(e)\rangle, 
	\end{align*}
where $d\in\Crec$ is a given proof of $\forall xA(x)$, $g(a)$ is an index of 
$ \lambda n. \weak(f(a), A(\bar n))$,
and $h(e)$ is an index of  $\lambda n.\weak(\vp(n),A(\bar n))$, where  the  partial recursive function $\vp$ is defined as
	\begin{align*}
\vp(0)&\simeq f(\kl e(0)), \\
\vp(n+1)&\simeq \begin{cases}
f(\kl e(n+1))  & \text{ if } \kl e(n)<'_\klo \kl e(n+1); \\
\text{undefined } & \text{ otherwise.}  \end{cases}
\end{align*} 
Recall that the relation $<'_\klo$ is recursively enumerable  and coincides with $<_\klo$ on $\klo$. By direct inspection, $f$ is primitive recursive.

By induction on $a\in\klo$, it is easy to show that $f(a)\in\Crec$ is a proof of $\forall x(x=x)$ of length $>|a|$. 	

For the other direction, note that if we simply define $h(e)$ as \[\Lambda n. \weak(f(\kl e(n)),A(\bar n)),\] then  $f(3\cdot 5^e)\in \Crec$ need not imply $3\cdot 5^e\in\klo$, even if $\kl e(n)\in\klo$ for every $n$. In fact, we also need $\kl e(n)<_\klo \kl e(n+1)$ for every $n$. This is why we use $\vp$ instead of $\lambda n.f(\kl e(n))$.

The proof proceeds by induction on the length of a code.  We only consider the limit case. Suppose $f(3\cdot 5^e)\in \Crec$. Then $h(e)$ is the index of a total recursive function and $\kl{h(e)}(n)\in\Crec$ for every $n$. It follows that $\vp$ is a total recursive function. This implies that $\kl e$ is  total  and  $\kl e(n)<'_\klo \kl e(n+1)$ for every $n$. It remains to show that  $\kl{e}(n)\in\klo$ for every $n$.  By definition, $\kl{h(e)}(n)\simeq \weak(f(\kl e(n)),A(\bar n))$.  By the properties of $\weak$, we thus have $f(\kl e(n))\in \Crec$. Now, $|f(\kl e(n))|<|f(3\cdot 5^e)|$. Note that $\weak$ preserves the length of a code. By  induction, it follows that $\kl e(n)\in\klo$, as desired.	
\end{proof}

\vspace{3pt}

The proof shows that a very tiny fragment of $\omega$-arithmetic, consisting of axioms and $\omega$-rule only, is already $\Pi^1_1$ complete.  An alternative reduction of $\klo$ to $\Crec$, although using cuts in a systematic way, is based on the following idea. Fix a true sentence $A$ and primitive recursively transform every $a\in\klo$ into a code $f(a)\in\Crec$ of an $\omega$-proof of $A$ ending as follows.
\[\begin{prooftree}
 \[  \[ \leadsto A \ \ \ A \ \ \cdots\ \  A\ \  \cdots \]  \justifies
A,\forall xA  \using{\orule} \]  \qquad \[ \[ g \leadsto A,\neg A \] \justifies A,\exists x \neg A  \using{\exists}\]  \justifies A \using{\cut}
\end{prooftree} \]
For this construction use a given code $g\in\Crec^-$ of an $\omega$-proof of $A,\neg A$. Note that for every sentence $B$ we can primitive recursively define a code $g(B)\in\Crec^-$ of an $\omega$-proof of $B,\neg B$. For instance, $g(\forall xB(x))$ encodes the following 
$\omega$-proof.
\[ \begin{prooftree}
\cdots \ \   \[ \[ g(B(\bar n)) \leadsto B(\bar n),\neg B(\bar n)\]  \justifies B(\bar n),\exists x\neg B(x) \using{\exists}\] \ \ \cdots 
\justifies \forall xB(x),\exists x\neg B(x) \using{\orule}
\end{prooftree} \]

\section{The primitive recursive $\omega$-rule}

Define $\Cpr$ as $\Crec$ but now use indices for primitive recursive functions. That is, replace the inductive clause for the $\omega$-rule with
\begin{itemize}
\item[($\omega$-rule)] if $\kl e$ is primitive recursive,  $\kl e(n)\in \Cpr$ and  $\End(\kl e(n))=\Gamma,A(\bar n)$ for every $n$, then 
$\langle \omega, \Gamma,\forall xA(x),e\rangle\in \Cpr$.
\end{itemize}
Of course, $\Cpr\subseteq \Crec$. 

\begin{remark} One can use a standard numbering of all primitive recursive functions. The two approaches are slightly different. The results of this paper apply either way. We will occasionally pinpoint the differences. 
\end{remark}



We make extensive use of the primitive recursive function $\weak$ performing weakening as defined in the proof of Theorem \ref{thm rec complete}. We list the main  properties without proof. These will be used without further notice.

\begin{lemma}
There is a primitive recursive function $\weak$ such that, for every $a$ and for every sequent $\Delta$,  
$a\in \Crec$ if and only if  $\weak(a,\Delta)\in\Crec$. The same property holds with respect to $\Crec^-$, $\Cpr$, and $\Cpr^-$. Also, if $a\in \Crec$ then  $\End(a)\cup \Delta=\End(\weak(a,\Delta))$ and $|a|=|\weak(a,\Delta)|$.
\end{lemma}

\subsection{Indices versus codes}\label{index vs code}
The set of global codes  of  (primitive) recursive $\omega$-proofs is defined according to Definition \ref{omega proof}.
\begin{definition}\label{indices}
	The set $\Irec$ ($\Ipr$) consists of all numbers  $a$ such that  $a$ is the index of a (primitive) recursive $\omega$-proof.
\end{definition}
Therefore, by definition, an $\omega$-proof $\pi$ is  (primitive) recursive if and only if there is a (primitive) recursive global code $a$ such that $\pi=\kl a$.  For $a\in\Irec$, let $\pi(a)=\kl a$ and $|a|=|\pi(a)|$. \smallskip

The following effective transformations between local codes and indices (global codes)   can be found in Sundholm \cite{S83}. The transformations involving primitive recursive codes   are obtained by allowing vacuous rules. These are  the repetition rule (cf.\ \cite{B91,M78})
\[  \rul{\Gamma}{\Gamma}{\rep}, \] and the vacuous $\omega$-rule  \[ \rul{\Gamma\ \ \Gamma \ \ \cdots\ \  \Gamma\ \  \cdots }{\Gamma}{\orep}. \] 
Define $\Crec^{\rep}$ and $\Crec^{\orep}$ by including the superscripted rule. Similar definitions apply to all other proof notation systems. 

All of the following transformations preserve the end sequent of an $\omega$-proof, although they do not necessarily preserve the $\omega$-proof itself.
\begin{itemize}
	\item[(r1)] There is a primitive recursive function $f$ such that if $a\in \Crec$, then  $f(a)\in\Irec$ is an index of $\pi(a)$, that is, $\pi(a)=\kl{f(a)}$. 
	\item[(r2)] There is a partial recursive function $\vp$ such that if  $a\in\Irec$, then   $\vp(a)$ is defined and $\vp(a)\in \Crec$ is a code of $\pi(a)$, that is, $\kl a=\pi(\vp(a))$.
	\item[(idx)] There is a primitive recursive function $f$ such that if  $a\in \Irec$, then $f(a)\in\Ipr^{\rep}\setminus\Ipr$ is an index of an $\omega$-proof $\pi(f(a))$ with same end sequent as $\pi(a)$. Moreover, $|a|\leq |f(a)|$ and, save for the repetition rule, every rule in $\pi(f(a))$ appears in $\pi(a)$.
\end{itemize}
\begin{itemize}
	\item[(pr1)]  There is a primitive recursive function $f$ such that if $a\in\Cpr$, then $f(a)\in\Ipr^{\rep}\setminus\Ipr$ is an index of an $\omega$-proof $\pi(f(a))$ with same end sequent as $\pi(a)$.  Moreover, $|a|\leq |f(a)|$ and, save for the repetition rule, every rule in $\pi(f(a))$ appears in $\pi(a)$.
	\item[(pr2)] There is  a partial recursive function $\vp$ such that if  $a\in\Ipr$, then $\vp(a)$ is defined and $\vp(a)\in\Cpr^{\orep}$ is a code of an $\omega$-proof $\pi(\vp(a))$ with same end sequent as $\pi(a)$.  Moreover, $|a|\leq |\vp(a)|$ and, save for the vacuous $\omega$-rule, every rule in $\pi(\vp(a))$ appears in $\pi(a)$.   
\end{itemize}

The vacuous rules in (idx), (pr1), and (pr2) can be replaced by cuts. For instance,
\begin{itemize}
\item[(idx)]  There is a primitive recursive function $f$ such that if  $a\in \Irec$, then $f(a)\in \Ipr\setminus \Ipr^-$ is an index of an $\omega$-proof $\pi(f(a))$ with same end sequent as $\pi(a)$. Moreover, $|a|\leq |f(a)|$ and, save for the cut rule, every rule in $\pi(f(a))$ appears in $\pi(a)$.	
\end{itemize}

Sundholm \cite{S83} works with a standard numbering $\{[e]\colon e\in\omega\}$ of primitive recursive functions and corresponding definitions of  $\Cpr$ and $\Ipr$. We do not. However, all relations above hold \emph{mutatis mutandis}. By using indices $[e]$'s, one can obtain a total recursive function in (pr2) instead of just a partial recursive one. \smallskip

Let us  give a proof of item (pr2) (the variant with cuts). 
\begin{theorem}\label{pr index to pr code}
There is a partial recursive function $\vp$ such that  $a\in \Ipr$ implies $\vp(a)$ is defined, $\vp(a)\in\Cpr$, and $\End(\kl a(\langle\rangle))=\End(\vp(a))$.
\end{theorem}
\begin{proof}
The idea is to  apply the following transformation when $R$ is a finitary rule.
	\[ \rul{\Gamma_1\ \ \cdots \ \ \Gamma_k}{\Gamma}{R} \ \  \ 	\rightsquigarrow  \ \ 
	\ptree{\[ \[\Gamma_i, \bar n\neq \bar n \ \ \  (n\in\omega) \j \Gamma_i,\forall x(x\neq x) \using{\orule} \]  \ \ \[\Gamma_i, 0=0 \j \Gamma_i,\exists x(x=x) \using{\exists} \] \j
		\Gamma_i \ \ \  (i\in |R|) \using{\cut} \] \j \Gamma \using{R} } \]
	
By the recursion theorem, we can find an index $e$ of a partial recursive function such that 
	\begin{align*}
	\kl e(a,\sigma) &\simeq \begin{cases}  \kl a(\sigma) & \text{ if } \kl a(\sigma)=\langle\ax, \Gamma\rangle;\\
	\langle R,\Gamma,A, b_0,\ldots,b_{k-1}\rangle & \text{ if }  \kl a(\sigma)=\langle R, \Gamma,A\rangle;\\
	\langle \omega, \Gamma,\forall xA(x),\Lambda n.\kl e(a,\sigma\conc n)\rangle & \text{ if } 
	\kl a(\sigma)=\langle \omega,\Gamma,\forall xA(x)\rangle,
	\end{cases}
	\end{align*}
	where 
\[   b_i=\langle \cut,\Gamma_i, \forall x(x\neq x), b_{i0}, b_{i1}\rangle, \text{ with } \Gamma_i=\End(\kl a(\sigma\conc i)), \]
\[   b_{i0}= \langle \omega,\Gamma_i,\forall x(x\neq x), c_i\rangle, \text{ with } c_i=\Lambda n.\weak(\kl e(a,\sigma\conc i),\bar n\neq\bar n), \]
\[   b_{i1}= \langle \exists, \Gamma_i,\exists x(x=x), \langle \ax,\Gamma_i\cup 0=0\rangle\rangle. \] 

The partial recursive function $\vp$ is given by $\vp(a)\simeq \kl e(a,\langle\rangle)$.  Suppose $a\in\Ipr$. We claim that $\kl e(a,\langle\rangle)\in \Cpr$ is as desired. By  induction on the length of $\pi(a)$,  show that $\kl e(a,\sigma)\in \Cpr$ is an $\omega$-proof of $\End(\kl a(\sigma))$. The crucial point is that if $a\in\Ipr$, then $\lambda \sigma.\kl e(a,\sigma)$ is primitive recursive as well. 
\end{proof}

\begin{corollary}
	$\thm(\Cpr)=\thm(\Crec)$. 
\end{corollary}
\begin{proof}
The  Sch\"{u}tte tree is primitive recursive and hence the Sch\"{u}tte tree of a true sentence has an index in $\Ipr$.
\end{proof}

The proof of Theorem \ref{pr index to pr code} is inspired by Sundholm \cite{S83}. Nonetheless, we do not use, quote, \lq\lq the full power of the recursion theorem\rq\rq.  That is, we do not exploit the fact that the recursion theorem holds uniformly in a primitive recursive way. This is because we are not using  a numbering $\{[a]\colon a\in\omega\}$ of the primitive recursive functions in defining  $\Cpr$ and $\Ipr$. We briefly outline how the proof would proceed in this setting. First, define a total recursive function $f(e,a,\sigma)$ corresponding to our $\kl e(a,\sigma)$  but use primitive recursive application $[e]$ and $[a]$ throughout. Note that $f$ is primitive recursive in the parameter $a$. That is,  $\lambda e\sigma.f(e,a,\sigma)$ is primitive recursive for every $a$.  The s-m-n theorem gives us a primitive recursive function $g$ such that $[g(a)](e,\sigma)=f(e,a,\sigma)$. By the effective primitive recursion theorem, there exists a primitive recursive function $h$ such that $[h(g(a))](\sigma)=[g(a)](h(g(a)),\sigma)$. Then the total recursive function $\lambda a.[h(g(a))](\langle\rangle)$ yields the desired transformation of $\Ipr$ into $\Cpr$. \smallskip

As for the relation between  primitive recursive $\omega$-proofs  and $\omega$-proofs with a primitive recursive local code, we have already mentioned that neither class is contained into the other.  In one direction, note that the Sch\"{u}tte tree of any sentence is primitive recursive (even elementary recursive). However, the Sch\"{u}tte tree  of a true sentence, although  encodable in  $\Crec$,  need not have a local code in $\Cpr$.  To see this, consider a true $\Pi_2$ sentence asserting the totality of a recursive, although not primitive recursive, function.  Then  the Sch\"{u}tte tree gives rise to a primitive recursive $\omega$-proof not encodable in $\Cpr$. For the other direction, consider the following recursive $\omega$-proof.  Let $f$ be a two-argument Ackermannian function so that $f$ is not primitive recursive, but the $n$-branch $\lambda m.f(n,m)$ is primitive recursive for every $n$.   Let $A(x,y,z)$ be the formula $x=x\land y=y\land z=z$. We leave  as an exercise to show that the following $\omega$-proof is not primitive recursive and yet it can be given a  code in $\Cpr$. 

\[\begin{prooftree}
\[  \cdots  \[ \[   
\[ \bar n=\bar n \qquad \bar m=\bar m  \justifies  \bar n=\bar n\land \bar m=\bar m \using{\land} \]   \qquad  f\overline{(n,m})=f\overline{(n,m})   \justifies A(\bar n,\bar m, f\overline{(n,m})) \using{\land} \] 
\justifies  \exists zA(\bar n,\bar m,z) \using{\exists} \]  \cdots 
\justifies  \cdots \qquad \forall y\exists zA(\bar n,y,z) \qquad \cdots \using{\orule} \] 
\justifies 
\forall x\forall y\exists zA(x,y,z) \using{\orule}
\end{prooftree}\]
(Hint: the function $\lambda n.(\Lambda m.f(n,m))$ is primitive recursive.)


\subsection{Completeness of the cut free primitive recursive $\omega$-rule}\label{pr complete}
We have  established  $\thm(\Crec)=\thm(\Cpr)$  by using the completeness of $\Ipr$ and the fact that one can transform the index of  a primitive recursive $\omega$-proof into a code of $\Cpr$. The price to pay for this index-to-code conversion is the introduction of new cuts. We do not know whether one can obtain such a conversion without cuts.  We aim to prove that the cut free fragment $\Cpr^-$ of $\Cpr$ is complete for true arithmetic.  

Let us first consider the propaedeutic $\Pi_2$ case ad show that  every true $\Pi_2$ sentence has a primitive recursive cut free $\omega$-proof encodable in  $\Cpr^-$.\footnote{Note by contrast that the intuitionistic case is  another matter altogether. In fact, primitive recursive cut free $\omega$-proofs based on Gentzen two-sided intuitionistic sequent calculus ($\mathsf{LJ}$)  are not even complete for Heyting arithmetic $\ha$. This  is well known (cf.\ \cite{LE76}). In fact,  consider a provably recursive function of $\ha$ (equivalently $\pa$) which is not primitive recursive, and  note that from any primitive recursive intuitionistic cut free  $\omega$-proof of $\forall x\exists yA(x,y)$ one can extract a primitive recursive function $f$ such that $\forall xA(x,f(x))$ is true. In particular, cut elimination does not hold unless we include vacuous rules such as, e.g., the repetition rule.}
Let $\forall x\exists yA(x,y)$ be a true  $\Pi_2$  sentence.  For $\sigma=\langle n_0,\ldots,n_l\rangle\in\seq$, let $\Gamma_\sigma$ be the sequent \[ \forall x\exists yA(x,y),\exists yA(\bar n_0,y),\ldots,\exists yA(\bar n_l,y).\] 
Let $g$ be a primitive recursive function such that $g(\sigma,m)\in\Cpr$ is the code of a cut free $\omega$-proof of the sequent $\Gamma_\sigma,A(\bar n_0,\bar m)$ whenever  $\sigma\in\seq$ and  $A(\bar n_0,\bar m)$ is true.

Define $h$ as
\begin{align*}
h(e,\sigma)&\simeq\begin{cases}
0    &   \text{ if }  \sigma\notin\seq; \\
\langle \exists, \Gamma_\sigma,\exists yA(\bar n_0,y), g(\sigma,m)\rangle &  \text{ if } \sigma(0)=n_0 \text{ and }  m<|\sigma|  \\
 & \text{ is least such that }A(\bar n_0, \bar m) \\
&  \text{ holds}; \\
\langle \omega,\Gamma_\sigma,\forall x\exists yA(x,y),\Lambda n.\kl e(\sigma\conc n)\rangle & \text{ otherwise. }
\end{cases}
\end{align*}
The function $h$ is primitive recursive. By the recursion theorem, there is an index $e$ such that $\kl e(x)\simeq h(e,x)$ for every $x$. Therefore $e$ is the index of a primitive recursive function. Let 
\[ a=\langle \omega,\forall x\exists yA(x,y),\kl e(\langle\rangle)\rangle.\]
Then $a$ is as desired. 

We now turn to the general case.  We will obtain  the completeness of the cut free fragment $\Cpr^-$  by  showing how to convert codes in $\Crec$ into codes in $\Cpr$ that preserve cut free $\omega$-proofs (Theorem \ref{thm pr complete}). 

We first need a lemma. Let $\N^\N$ consist of all total functions from $\N$ to $\N$. 
\begin{definition}
For every $a\in\N$ and $f\colon \N\to \N$ define 
$a(f)=\langle a_0,a_1,\ldots\rangle\in\seq\cup\N^\N$ such that $a_0=a$ and, for every $i$,
\begin{itemize}
\item if $a_i$ is $\langle \land,\Gamma,A\land B,b_0,b_1\rangle$  or $\langle\cut, \Gamma,C,b_0,b_1\rangle$, then $a_{i+1}=b_{f(i)(\md 2)}$,
\item if $a_i$ is $\langle\lor, \Gamma,A\lor B,b\rangle$, then $a_{i+1}=b$,
\item if $a_i$ is $\langle\omega, \Gamma,\forall xA(x),e\rangle$ and $\kl e(f(i))$ is defined, then $a_{i+1}=\kl e(f(i))$,
\item $a_{i+1}$ is undefined and $a(f)=\langle a_0,\ldots,a_i\rangle$ otherwise.
\end{itemize}
\end{definition}

\begin{lemma}\label{lemma derivation}
If $a\in\Crec$ then $a(f)\in\seq$ for every  $f\colon \N\to \N$.  
\end{lemma}
\begin{proof}
The proof is by induction on the code. 
\end{proof}

\begin{theorem}\label{thm pr complete}
There exists a primitive recursive function $f$ such that: 
\begin{enumerate}[(i)]
	\item $a\in\Crec$ if and only if $f(a)\in \Cpr$;
	\item $a\in\Crec^-$ if and only if $f(a)\in \Cpr^-$;
	\item $a\in \Crec$ implies $\End(a)=\End(f(a))$.
\end{enumerate}	
\end{theorem}

The idea is to delay the application of a rule preceding an $\omega$-inference by iterated applications of the $\omega$-rule. 
\[ \ptree{\[ \Gamma_i\ \ (i\in |R|)\ \  \j  \Gamma, A(\bar n)\ \  (n\in\omega) \using{R} \] \j \Gamma,\forall xA(x) \ \ \using{\orule}} \ \  \ 	\rightsquigarrow  \ \ \ 
\ptree{  \[\[\[\Gamma_i,\forall xA(x), A(\bar u),A(\bar v) \ \ (i\in|R|) \j 
	        \Gamma,\forall xA(x),A(\bar n),A(\bar u),A(\bar v)\ \ \ (v\in\omega) \using{R} \] \j
	      \Gamma,\forall xA(x),A(\bar n),A(\bar u) \ \ (u\in\omega)  \using{\orule} \] \j
	      \Gamma,\forall xA(x),A(\bar n)  \ \ (n\in\omega)  \using{\orule} \] \j
	      \Gamma,\forall xA(x)  \using{\orule}} \]
Suppose that the proof on the left-hand side has code  $\langle \omega, \Gamma,\forall xA(x),a\rangle$ in $\Crec$. Then for every $n$ we  iterate the $\omega$-rule until $\kl a(n)$ converges. For the sake of local correctness, we repeat the main formula $\forall xA(x)$.   Once we know $\kl a(n)$ we keep working on the proof coded by $\kl a(n)$. If it does not converge, the resulting (sub)proof  of $\Gamma,\forall xA(x),A(\bar n)$ on the right-hand side is ill founded but still locally correct. 

\begin{proof}
By Kleene's normal form theorem, there is a primitive recursive predicate $T$ such that
\[   \kl a(n)\simeq U(\mu z.T(a,n,z)), \]
where $U$ is primitive recursive. 

By the recursion theorem we can find a primitive recursive function $g$ such  that
\[ g(e,\Gamma,\Delta,\forall xA(x),a,n,z, u)\simeq \]
\begin{align*}
&\simeq  
\begin{cases}
\weak(\kl e(U(z)),\Delta\cup A(\bar u) \setminus A(\bar n))   \\
\qquad \text{ if } T(a,n,z) \text{ and } \End(U(z))=\Gamma,A(\bar n); \\
\langle \omega, \Delta\cup A(\bar n)\cup A(\bar u),\forall xA(x),\Lambda v.g(e,\Gamma,\Delta\cup A(\bar u),\forall xA(x), a,n,z+1, v)\rangle  \\
\qquad \text{ otherwise.}
\end{cases} 
\end{align*}

By another application of the recursion theorem, there is an index $e$ such that 
\begin{align*}
\kl e(\langle \ax,\Gamma\rangle) &\simeq \langle\ax, \Gamma\rangle,  \\
\kl e(\langle R,\Gamma,A,a_1,\ldots,a_n\rangle) &\simeq \langle R,\Gamma,A,\kl e(a_1),\ldots,\kl e(a_n)\rangle, \\
\kl e(\langle \omega,\Gamma,\forall xA(x),a\rangle)& \simeq \langle\omega, \Gamma\cup\forall xA(x),\forall xA(x),b \rangle,
\end{align*}
with 
\[ b= \Lambda n.g(e,\Gamma, \Gamma\cup\forall xA(x), \forall xA(x),a,n, 0, n). \]

By direct inspection, the index $e$ defines a primitive recursive function $f$.	

The proof that $f$ is as desired is by induction on (the length of) the codes.  We show items (i) and (iii). Item (ii) is analogous. The only case that requires some attention is when we are given a code $c$ of the form  $\langle\omega, \Gamma,\forall xA(x),a\rangle$. By definition,
\[  f(c)=\langle \omega,\Gamma\cup\forall xA(x),\forall xA(x),b\rangle, \]
where $b$ is as above. Note that $b$ is the index of a primitive recursive function.

Suppose \sloppy $c=\langle\omega, \Gamma,\forall xA(x),a\rangle\in\Crec$. Then $\kl a$ is total  and $\kl a(n)\in\Crec$ with end sequent $\Gamma,A(\bar n)$ for every $n$. By the induction hypothesis, $f(\kl a(n))\in \Cpr$ with end sequent $\Gamma,A(\bar n)$ and so $\weak(f(\kl a(n)),\Delta)\in\Cpr$ with end sequent $\Gamma,\Delta,A(\bar n)$ for every $\Delta$ and every $n$. Therefore it is sufficient to show that for every $n$ we have $g(e,\Gamma, \Gamma\cup\forall xA(x),\forall xA(x),a,n,0,n)\in \Cpr$ with end sequent $\Gamma,\forall xA(x),A(\bar n)$. Fix $n$ and let $z$ be least such that $T(a,n,z)$. Prove by induction on $i\leq z$ that for every $\Delta\supseteq \Gamma,\forall xA(x)$ and for every $u$
\[    g(e,\Gamma, \Delta,\forall xA(x), a,n,z-i,u)\in \Cpr \]
with end sequent $\Delta, A(\bar n),A(\bar u)$. For $i=0$, we have
  \[ g(e,\Gamma, \Delta,\forall xA(x), a,n,z,u)\simeq \\  \weak(f(\kl a(n)), \Delta\cup A(\bar u)\setminus A(\bar n)). \] 
This case follows by the above assumption on $f(\kl a(n))$. The case with $i+1\leq z$ follows by the induction hypothesis on $i$. In fact,
\begin{multline*} g(e,\Gamma, \Delta,\forall xA(x), a,n,z-i-1,u)\simeq \\ \langle \omega, \Delta\cup A(\bar n)\cup A(\bar u),\forall xA(x),\Lambda v.g(e,\Gamma,\Delta\cup A(\bar u),\forall xA(x), a,n,z-i, v)\rangle. \end{multline*}

For $i=z$, $\Delta=\Gamma\cup\forall xA(x)$ and $u=n$, we obtain the desired result. 

For the converse direction, let $c=\langle\omega, \Gamma,\forall xA(x),a\rangle$ and suppose $f(c)\in\Cpr$. We aim to show that $\kl a$ is total  and $\kl a(n)\in \Crec$ with end sequent $\Gamma,A(\bar n)$ for every $n$.  Fix $n$. We know that $\kl b(n)$ is defined and  $\kl b(n)\in\Cpr$. We claim that $\kl a(n)$ is defined  with $\End(\kl a(n))= \Gamma,A(\bar n)$. Suppose not.  Fix any function $h\colon \N\to \N$.  Then $(\kl b(n))(h)=\langle b_{0},b_{1}\ldots \rangle$ is infinite and every $b_i$ is of the form $\langle \omega,\Delta,\forall xA(x),c_i\rangle$. This is seen by induction on $i$. By Lemma \ref{lemma derivation}, we obtain that $\kl b(n)\notin\Cpr$, a contradiction. 
Hence $\kl a(n)$ is defined and $\End(\kl a(n))=\Gamma,A(\bar n)$ for every $n$.  It remains to show that $\kl a(n)\in \Crec$.  By construction (see the first clause in the definition of $g$), there is a $\Delta$ such that $\weak(f(\kl a(n)), \Delta)$ is a subcode of $f(c)$. Therefore $\weak(f(\kl a(n)), \Delta)\in\Cpr$ and so is  $f(\kl a(n))$. Since $|f(\kl a(n))|=|\weak(f(\kl a(n)), \Delta\})|<|f(c)|$, we can apply the induction hypothesis and conclude $\kl a(n)\in \Crec$, as desired. \end{proof}

Note that for every $a\in\Crec$,  $|a|\leq |f(a)|$,  and every rule appearing in $\pi(f(a))$ already occurs in $\pi(a)$.

\begin{corollary}
For every true arithmetical sentence $A$ there is primitive recursive local code $a\in\Cpr^-$ such that $\pi(a)$ is a primitive recursive cut free $\omega$-proof of $A$. Therefore, $\thm(\Crec)=\thm(\Cpr^-)$  coincides with the set of true arithmetical sentences.
\end{corollary}
\begin{proof}
It follows from Theorem \ref{shoenfield} and Theorem \ref{thm pr complete}.
The only part that requires some explanation is the claim  that the $\omega$-proof encodable in $\Cpr^-$ is indeed primitive recursive. In fact, we have already observed that primitive recursive local codes do not correspond to primitive recursive $\omega$-proofs. Now, by Theorem \ref{shoenfield}, a true sentence $A$ has a primitive recursive cut free $\omega$-proof $\pi$ with a local code $a$ in $\Crec^-$. It follows by Theorem \ref{thm pr complete} that $f(a)\in\Cpr^-$ is a local code of some recursive $\omega$-proof $\pi(f(a))$ of $A$.  The point is that the $\omega$-proof $\pi(f(a))$ is primitive recursive for every $a\in\Crec$ (be $\pi(a)$ primitive recursive or not).  We leave this as an exercise for the reader.  
\end{proof}

\begin{corollary}
$\Cpr$ and $\Cpr^-$ are $\Pi^1_1$ complete.
\end{corollary}
\begin{proof}
It follows from Theorem \ref{thm rec complete} and Theorem \ref{thm pr complete}.
\end{proof}


\end{document}